\documentclass[envcountsect,runningheads,a4paper]{llncs}
\usepackage{amsmath,amssymb, mathrsfs} 
\usepackage{amsfonts}
\usepackage[T1]{fontenc}
\usepackage{fancyhdr}
\usepackage[skip=1pt plus1pt, indent=10pt]{parskip}
\usepackage{xspace}
\usepackage{xcolor}
\usepackage{amssymb}
\usepackage{amscd}
\usepackage{mathrsfs}

\usepackage{tikz}

\usetikzlibrary{matrix}
\usetikzlibrary{positioning}
\usetikzlibrary{shapes.geometric}
\usetikzlibrary{shapes,decorations,shadows}  
\usetikzlibrary{decorations.pathmorphing}  
\usetikzlibrary{decorations.shapes}  
\usetikzlibrary{fadings}  
\usetikzlibrary{patterns}  
\usetikzlibrary{calc}  
\usetikzlibrary{automata}
\usetikzlibrary{chains,fit,shapes}
\usetikzlibrary{backgrounds}

\def\beq{\arraycolsep1pt\begin{eqnarray*}}
\def\eeq{\end{eqnarray*}}

\newcommand{\br}{\mathbb{R}}

\newcommand{\C}{\mathbb{C}}
\newcommand{\bd}{\bar{\partial}}
\newtheorem{Thm}{Theorem}[section]
\newtheorem{lem}[Thm]{Lemma}
\newtheorem{pro}[Thm]{Proposition}
\newtheorem{de}[Thm]{Definition}
\newtheorem{re}[Thm]{Remark}
\newtheorem{ex}[Thm]{Example}

\spnewtheorem{thm}{Theorem}[section]{\bfseries}{\itshape}
\def\be{\begin{equation}}
\def\ee{\end{equation}}
\def\bea{\begin{eqnarray}}
\def\eea{\end{eqnarray}}
\numberwithin{equation}{section}

\usepackage[colorlinks,linkcolor=blue]{hyperref}

\usepackage{geometry}
\usepackage[utf8]{inputenc}

\usepackage{graphicx} 
\renewenvironment{abstract}
 {\small
  \vspace{20pt}
  \list{}{
    \setlength{\leftmargin}{2cm}%
    \setlength{\rightmargin}{\leftmargin}%
  }%
  \item\relax}
 {\endlist}

\begin{document}
\newgeometry{
  a4paper,        
  textwidth=14.5cm,  
  textheight=20cm, 
  heightrounded,   
  hratio=1:1,      
  vratio=5:2,      
}

\title{Collapsing of Mean Curvature Flow of Hypersurfaces\\ 
to Complex Submanifolds
}

\author{Farnaz Ghanbari\inst{1} \and Samreena\inst{2}}

\institute{Tarbiat Modares University, Tehran, Iran and ICTP, Trieste, Italy \and SISSA, Trieste, Italy}

\maketitle
\thispagestyle{plain}

\begin{abstract}
\textbf{Abstract.}
In this paper, we produce explicit examples of mean curvature flow of $(2m-1)$-dimensional submanifolds which converge to $(2m-2)$-dimensional submanifolds at a finite time. These examples are a special class of hyperspheres in $\mathbb{C}^{m}$ with a $U(m)$-invariant K\"ahler metrics.  We first discuss the mean curvature flow problem and then investigate the type of singularities for them.

\end{abstract}

\section{Introduction}

\hspace{10pt} Mean curvature flow is a well-known geometric evolution equation for hypersurfaces in which each point moves with a velocity given by the mean curvature vector. If the hypersurface is compact, the short time existence and uniqueness of the mean curvature flow are well-known. In general, it is very hard to find an exact solution of mean curvature flow problem. In fact there are very few explicit examples. Round spheres in Euclidean space are non trivial examples of evolving hypersurface  under mean curvature flow which concentrically shrink inward until they collapse at a finite time to a single point. Another instance would be the marriage ring that under mean curvature flow shrinks to a circle. A round cylinder also remains round and finally converges to a line. Mean curvature flow develops singularities if the second fundamental forms of the time dependent immersions become unbounded.
 It is well-known that mean curvature flow of any closed manifold in Euclidean space develops singularities at a finite time.

 The mean curvature ﬂow has ﬁrst been investigated by Brakke in  1978 \cite{Brakke}. Later Huisken \cite{Huisken: 1984} showed that any closed convex hypersurface in Euclidean space shrinks to a round point at a finite time. He then proved \cite{Huisken:1986} that the same holds for hypersurfaces in general Riemannian manifolds satisfying a strong convexity condition which takes into account the geometry of the ambient space. Brakke used geometric measure theory,
but Huisken employed a more classical diﬀerential geometric approach. In order to describe singularities of the ﬂow, Osher-Sethian introduced a level-set formulation for the mean curvature ﬂow which was investigated later by Evans-Spruck (\cite{Evans},  \cite{Evans1},  \cite{Evans2},  \cite{Evans3}) and Chen-Giga-Goto \cite{Chen} in details. Ilmanen revealed in \cite{Ilmanen} the relation between the level-set formulation and the geometric measure theory approach.

In this paper, we consider a class of canonical hyperspheres in $\mathbb{C}^{m}$.
We will make an important assumption about the symmetry group. Namely, we will require that the K\"ahler metric on $\C^m\setminus {0}$ has $U(m)$ as
the group of isometries. We study the mean curvature flow problem for hyperspheres in $Bl_{0}\C^m$  which reduce to an ordinary differential equation due to invariance of the metric and mean curvature under isometries. In general, it is not easy to compute the second fundamental form to investigate the singularities and also their types. We have computed all the principal curvatures and observed that near the exceptional divisor, all the principal curvatures vanish except for one direction which goes to infinity. By knowing the principal curvatures, we 
\newpage
\restoregeometry

\hspace{-.37cm}can compute the mean curvature and also the square of the norm of the second fundamental form.  There is a known example, Burns metric on $Bl_{0}\C^2$, for which we will study the mean curvature flow problem  in Section \ref{sec:5} and show the exact time of singularity.

The  rest of the paper is organized as follows. Section \ref{sec:2} is devoted to definitions, some well-known theorems and some results that will be used throughout the work. Section \ref{sec:3} focuses on the blow up of $\C^m$ at the origin. We discuss the condition when a $U(m)$-invariant metric on $\C^m\setminus {0}$ can extend to the blow up of $\C^m$ at the origin. 
In Section \ref{sec:4}, we state and prove a proposition on computing principal curvatures on special cases that leads to the proof of our main theorem. Finally Section \ref{sec:5} is dedicated to the mean curvature flow for our setting and some examples.

\section{Preliminaries}
\label{sec:2}

\hspace{10pt} In this section, we recall some basic notions and fix some notations used throughout this paper.

\begin{de}
Let $F_{0}: \Sigma^{m} \longrightarrow M^{m+1} $ be a smooth immersion of an $m$-dimensional manifold. The mean curvature flow of $F_{0}$ is a family of smooth immersions $F_{t}: \Sigma \longrightarrow M^{m+1}  $ for $t\in [0,T)$ such that setting $F(p,t)=F_{t}(p)$ the map 
$ F: \Sigma \times [0,T) : \Sigma^{m} \longrightarrow M^{m+1} $ is a smooth solution of the following system of PDE's

\begin{equation*}\label{eq:in:introduction 1}
     \begin{cases}
\frac{\partial}{\partial{t}} F(p,t)=H(p,t)n(p,t),\\
F(p,0)=F_{0}(p) \,,
\end{cases} 
\end{equation*}
where $H(p,t)$ and $n(p,t)$ are respectively the mean curvature and the unit normal of the hypersurface $F_{t}$ at the point $p\in \Sigma$.
\end{de}

Usually the Riemannian manifold $M$ is called the ambient manifold and the parameter t is considered as time. Minimal submanifolds, i.e. submanifolds with zero mean curvature
everywhere, are the stationary solutions of this flow.  

There are two important propositions in the Euclidean case. We use them in the proof of our main theorem on mean curvature flow problem \cite{Mantegazza}. The propositions are as follows.
 
\begin{pro}
\label{Pro:1}
 If the second fundamental form is bounded in the interval $[0,T)$ with $T< +\infty$, then all its covariant derivatives are also bounded.
\end{pro}
\medskip

\begin{pro}
\label{Pro:2}
    If the second fundamental form is bounded in the interval $[0,T)$ with $T< +\infty$, then $T$ cannot be a singular time for the mean curvature flow of a compact hypersurface $F: \Sigma \times[0,T) \longrightarrow \mathbb{R}^{n+1}$.
\end{pro}
\medskip

From these two propositions, we have the following Remark.
\medskip

\begin{re}
    The above estimate can be found independent of $T$ and also independent of initial data.
\end{re}
\medskip

One of the most important problems in studying the mean curvature flow is to understand the possible singularities the flow goes through. We introduce the notion of singularity in mean curvature flow and their types in the following.
\medskip
\begin{de}
    If the second fundamental form $|A|^2$ blows up at $t\longrightarrow T$, then we call $T$ a singular time of the flow.
\end{de}
\medskip

\begin{de}
 We say that the flow is developing a type $I$ singularity  at time $T$ if there exists a constant $C>1$ such that we have the upper bound 
    $$ max_{p\in \Sigma}|A(p,t)|^2 \leq \frac{C}{T-t}.$$
    Otherwise, we say it is a type $II$ singularity.
    
\end{de}

\section{K\"ahler Metrics on
the Blow Up of $\C^m$ at the Origin}
\label{sec:3}

\hspace{10pt} We consider the blow up of  $\C^m$ at the origin and denote it by $Bl_0\C^m$. It is defined as following:
$$Bl_0\C^m = \{((z_1,z_2,\dots,z_m),[t_1,t_2,\dots,t_{m}]) \in \C^m \times \C P^{m-1} : z_it_j-z_jt_i=0\} \subset\C^m \times \C P^{m-1}\,.$$

There is a natural projection map $\pi_1 : Bl_0\C^m \rightarrow{\C^m}$ defined by
\begin{equation*}
    \pi_1((z_1,z_2,\dots,z_m),[t_1,t_2,\dots,z_{m}])=(z_1,z_2,\dots,z_m)\,.
\end{equation*}

The inverse image $\pi^{-1}_1(p)$ of   $p \in \C^{m}$  is a line passing the point $p$.

The \textbf{exceptional divisor} $E$ is defined as the inverse image of the origin i.e., $\pi^{-1}(0)= \C P^{m-1}.$

Moreover the map $\pi_{1}$ can be restricted to a biholomorphism
\begin{equation*}
    \pi_1:  Bl_0\C^m \setminus E \rightarrow{ \C^m\setminus 0}.
\end{equation*}

A system of charts that covers the exceptional divisor is given as follows: for every $i=1,2,\dots,m$,
$$U_{i} = \{((z_1,z_2,\dots,z_m),[t_1,t_2,\dots,z_{m}]): t_i\not=0, z_j=z_it_j \}\,.$$

The coordinate map $     \Phi_i:U_{i} \rightarrow{}\C^m $
is defined as
\begin{equation*}
    ((z_1,z_2,\dots,z_m),[t_1,t_2,\dots,t_{m}])\to \left(z_i, \frac{t_1}{t_i},\dots,\frac{t_{i-1}}{t_i},\frac{t_{i+1}}{t_i},\dots,\frac{t_m}{t_i}\right),
\end{equation*}
 with inverse map $ \Phi_{i}^{-1}:\C^m\rightarrow{U_{i}}$
 \begin{equation}\label{coordinate}
     (z_1,z_2,\dots,z_m)\to ((z_1z_{i},z_{i} z_{2},\dots,z_{i},\dots,z_{i}z_{m}),[z_{1},\dots,z_{i-1},1,z_{i+1},\dots,z_{m}]).
 \end{equation}
 
For every $i=1,2,\dots,m,$ the chart $U_{i}$ intersects the exceptional divisor $E$:
 $$ E \cap U_{i}=\{ z_{i}=0\}\,. $$
 
We now take the smooth $(1,1)$-form on $\C^m\setminus 0$ given by 
$$\omega = \sqrt{-1} \partial \bd \log (S),$$
where $S= \sum^m_{i=1}|z_i|^2$.

The pull back of the smooth form 
$ \omega = \sqrt{-1} \partial \bd \log (S)$ on $\C^m\setminus{0}$ extends to the Fubini Study metric on the exceptional divisor $E=\C P^{m-1}$.

The pull back $\pi^{*}_{1}\omega$ is given in local coordinates (\ref{coordinate}) by,
\begin{align} \label{ metric on E}
    \pi^{*}_{1}\omega &=\partial \bd \log(|z_i|^2(|z_{1}^2|+|z_2|^2\dots+|z_{i-1}|^2+1+|z_{i+1}|^2+\dots+|z_m|^2)\nonumber\\ 
    &= \partial \bd \log(|z_{1}|^2+|z_2|^2\dots+|z_{i-1}|^2+1+|z_{i+1}|^2+\dots+|z_m|^2).
\end{align}

Clearly (\ref{ metric on E}) is the Fubini Study metric on the exceptional divisor $E$ in homogeneous coordinates $[z_{1},\dots,z_{i-1},1,z_{i+1},\dots,z_{m}].$

 Let $g:\C^m\to \br$ be a smooth function that depends on $S=\sum^m_{i=1}|z_i|^2$. Then the smooth form 
\begin{equation}\label{K\"ahler: form}
    \omega =\sqrt{-1}\partial \bd f(S)= \sqrt{-1}\partial \bd( \log S+g(S))\
\end{equation}
gives K\"ahler metric on $\C^m\setminus \{0\}$ if and only if $\frac{1}{S} +g_{S}>0$ and $ g_{S}+Sg_{SS}>0$. The next proposition explains 
when the K\"ahler form \eqref{K\"ahler: form} on $\C^m\setminus{0}$ can be extended to $Bl_0\C^m$.
\medskip

\begin{pro}\label{form :extend to blow up}
The smooth form $\omega =\sqrt{-1} \partial \bd( \log S+g(S))$ on $\C^m\setminus \{0\}$ extends to K\"ahler metric on $Bl_0\C^m$  if and only if $g_{S}(0)>0$, $\frac{1}{S} +g_{S}>0$ and $ g_{S}+Sg_{SS}>0$. 
\end{pro}

\begin{proof}
For the sake of simplicity, we only prove the case when $m=2.$ The general case follows from the same argument.

Given the projection map
$$ \pi_{1} : Bl_0\C^2\rightarrow{}\C^2 ,$$
on the chart $U_{1}$ we have $S=|z_{1}|^2(1+|z_{2}|^2)$ and $E \cap U_{1}=\{z_1=0 \}\,.$
The pull back of the K\"ahler metric  (\ref{K\"ahler: form}) to $Bl_0\C^2$ is given in coordinates (\ref{coordinate}) by
\[
\pi_{1}^*\omega= 
\begin{bmatrix}(1 + |z_2|^2)(g_{S}+ S g_{SS})& z_1	\bar{z_2}(g_{S}+ S g_{S S})\\
z_2	\bar{z_1}(g_{S}+ S g_{SS}) & |z_1|^2 (g_{S} +  |z_1|^2|z_2|^2 g_{SS})+ \frac{1}{1+|z_2|^2}
\end{bmatrix}.
\]

The  restriction of  $\pi_{1}^*\omega$ to the exceptional divisor $E$ is:
\[
\pi_{1}^*\omega|_{E} = 
\begin{bmatrix}(1 + |z_2|^2)g_{S}(0)& 0\\
0& \frac{1}{1+|z_2|^2}
\end{bmatrix}.
\] 

Clearly $\pi_{1}^*\omega|_{E}$ is positive definite if and only if $g_{S}(0)>0$.

In the same way on  $U_{2}$, the pull back  $\pi_{1}^*\omega$ where
\[
\pi_{1}^*\omega = 
\begin{bmatrix}\frac{1}{1+|z_1|^2} + |z_2|^2 (g_{S} + |z_1|^2|z_2|^2g_{SS})& z_1	\bar{z_2}(g_{S}+ S g_{SS})\\
z_2	\bar{z_1}(g_{S}+ S g_{SS})& (1+|z_1|^2 )(g_{S}+ S g_{SS})
\end{bmatrix}
\] 
 can be restricted to the exceptional divisor as follows:
\[
\pi_{1}^*\omega|_{E} = 
\begin{bmatrix}\frac{1}{1+|z_1|^2} & 0\\
0& (1+|z_1|^2 )g_{S}
\end{bmatrix}
\] 
$\pi_{1}^*\omega|_{E} $ is positive definite if and only if $g_{S}(0) > 0 $. 
\end{proof}\qed
\medskip

\begin{re}
    If $g_{S}(0)=0$, then $\pi_{1}^*\omega|_{E}$ defines a metric only along the exceptional divisor. Therefore the condition $g_{S}(0)\neq 0$ guarantees the non degeneracy of the metric orthogonal to the exceptional divisor. The other two conditions $\frac{1}{S} +g_{S}>0$ and $ g_{S}+Sg_{SS}>0$ are considered because $\omega$ must be a K\"ahler metric on $\C^m\setminus{0.}$
\end{re}

\section{Principal Curvatures of Hyperspheres}
\label{sec:4}
\hspace{10pt} In this section, we compute the second fundamental form for hyperspheres under special conditions. In order to investigate the mean curvature flow for our examples, we need to know the principal curvatures which are the eigenvalues of the second fundamental form.

Let $\Sigma$ be an $d$-dimensional smooth submanifold in an $d+1$-dimensional manifold $M$ and $g$ be the Riemannian metric on $M$ with Levi Civita connection $\nabla$.
\medskip 
\begin{de}
The second fundamental form of $\Sigma$ is defined by 
\begin{equation}\label{eq:coefficient :of:2nd form}
    \Pi_{n}(X,X)=g\left(\nabla_{X}(X),n\right)\,,
\end{equation}
where $X \in T_{p}M$ and $n \in (T_{p}\Sigma)^{\perp}.$ 
\end{de}
\medskip 

\begin{lem} \label{2nd fundamental form}
    Suppose $X$ and $n$ are local vector fields on M such that 
    \begin{itemize}
        \item[1.] $||X||^2_{g}$ and $||n||^2_{g}$ are constants\,,
        \item[2.] for all $p \in \Sigma$, $X(p) \in T_{p}\Sigma$ and $n(p) \in (T_{p}\Sigma)^{\perp}$,
    \end{itemize}
    then $$\Pi_{n}(X,X)=-g([X,n],X).$$
\end{lem}

\begin{proof}
    \begin{align*}
        \Pi_{n}(X,X)& =g\left(\nabla_{X}(X),\eta\right) = -g(\nabla_{X}(n),X )
        = -g([X,n]+\nabla_n(X),X)\\
        &= -g([X,n],X) +\frac{1}{2} n(||X||_{g}) =-g([X,n],X).
    \end{align*}
\end{proof}\qed

In the next proposition, we state and prove the main result of this section and calculate the second fundamental from for hyperspheres with some particular assumptions.
\medskip
\begin{pro} \label{2nd fundamental form 1}
    Suppose that $g_{0}$ and $g$ are Euclidean and Riemannian 
    metrics on M respectively. Let $e_{1},...,e_{d+1}$ be orthonormal local vector fields for $M$ with respect to $g_{0}$ i.e., $g_{0}(e_{i},e_{j})=\delta_{ij}$. $\Sigma \subset M$ is an $m$-dimensional submanifold such that for each $p \in \Sigma$ we have $e_{d+1}(p) \perp  T_{p}\Sigma $. Let $n=e_{d+1}$ and $A,\eta,\mu $ be local functions on $M$ such that their restrictions on $\Sigma$ are constants. We have the following conditions:
   
    \begin{itemize}
        \item[1.]  $g(e_{d+1}, e_{d+1}) = A^{2}$ , $g(e_{d}, e_{d})= \mu^{2}$, 
        
        \item[2.] $g(e_{i},e_{i})=\eta^{2}$ \quad if $1 \leq i \leq d-1$
       \item[3.] $g(e_{i},e_{j})=0$  \quad    $\forall i \neq j$ 
       \item[4.] $[e_{d},e_{d+1}] \in  \mathbb{R}<e_{d},e_{d+1}>.$
       
    \end{itemize}
Now if $\Pi_{\Sigma}(g_{0})=\tau g_{0}$ for some $\tau \in \mathbb{R}$,  then
    
    \[ \Pi_{g}(e_{i},e_{j})=
\begin{bmatrix} (\eta^{2}A^{-1}\tau+ \eta A^{-1}\nabla_{n}\eta)I_{m-1}& 0\\
0 & \mu^{2}A^{-1}\tau + \mu A^{-1}\nabla_{n}\mu
\end{bmatrix}.
\]
\end{pro}
\begin{proof}
We fix some notations which will be used in the proof.

Let $[n,e_i]= \sum_{j=1}^{d+1} a_{ij}e_{j}$ for $1\leq i\leq d$. Then,
\begin{itemize}
    \item[$\bullet$] $a_{ii}= g_{0}([n,e_{i}],e_{i})=  -g_{0}([e_i,n],e_{i}) = \Pi_{g_{0}}(e_{1},e_{1}) = \tau$ 
    \item[$\bullet$] $ 0=2\Pi_{g_{0}}(e_i,e_{j}) =g_{0}  \nonumber([e_{i},n],e_{j})+g_{0}([e_{j},n],e_{i}) =a_{ij}+a_{ji}.$
     
\end{itemize}
\medskip

Notice that $\{\eta^{-1}e_{1},...,\eta^{-1}e_{d-1} ,\mu^{-1}e_{d},A^{-1}n \}$ is an orthonormal frame for the metric $g$. We prove the proposition in the following steps.
\medskip

\textbf{Step 1}: By the Lemma \ref{2nd fundamental form} we have
\begin{align*}
    \Pi(\eta^{-1}e_{1},\eta^{-1}e_{1})&=-g([\eta^{-1}e_{1}, A^{-1}n],\eta^{-1}e_{1} )\\
    &= -g(\eta^{-1} A^{-1}[e_1,n]-A^{-1}\nabla_{n}(\eta^{-1}e_1), \eta^{-1}e_{1} )\\
    &= -\eta^{-1} A^{-1}g([e_1,n],  \eta^{-1}e_{1} ) + \eta^{-1}A^{-1}\nabla_{n}(\eta^{-1}) g(e_1, e_{1} )\\
    &=-\eta^{-2} A^{-1}g([e_1,n],  e_{1} ) + A^{-1} \eta \nabla_{n}(\eta^{-1})\\
    &= -\eta^{-2} A^{-1}g([e_1,n],  e_{1} ) - A^{-1} \eta^{-1} \nabla_{n}(\eta)\\
    &=- A^{-1}(\tau + \eta^{-1}\nabla _{n}\eta)
\end{align*}
where we employed the property of Lie bracket and the fact that $\nabla_{e_{1}}A^{-1}=0$ on $\Sigma$. For the last step we use the following relation:
$$g([e_1,n],  e_{1})= a_{11} g_{11}=\tau \eta^2.$$

\textbf{Step 2}: The same calculation shows that 

$$\Pi(\eta^{-1}e_{i},\eta^{-1}e_{i}) = A^{-1}(\tau + \eta^{-1}\nabla _{n}\eta),$$ 
if $1 \leq i \leq d-1$,

$$\Pi(\mu^{-1}e_{d},\mu^{-1}e_{d}) = g([\mu^{-1} e_{d}, A^{-1}n], \mu^{-1}e_{d}).$$

Similar to the step 1 we have:

$$\Pi(\mu^{-1}e_{d},\mu^{-1}e_{d}) = A^{-1}(\tau + \mu^{-1}\nabla _{n}\mu).$$

Now similar to the last calculations we get $\Pi(e_{i},e_{j})=0$ for each $1 \leq i < j\leq d-1$.

In the next step we show that $\Pi (\eta^{-1}e_{1},\mu^{-1}e_{d})=0$.
\smallskip

\textbf{Step 3}:
\[
2\Pi(\eta^{-1}e_{1}, \mu^{-1}e_{d})= g([\eta ^{-1}e_{1},n], \mu^{-1}e_{d})+ g([\mu^{-1}e_{d},n], \eta^{-1}e_{1})
\]

We have: $$[\eta^{-1}e_{1},n]= \eta^{-1}[e_{1},n]- \nabla_{n}\eta^{-1}e_{1}=
\eta^{-1}\Sigma a_{ij}e_{j} - \nabla_{n}\eta^{-1}e_{1},$$

and $$[\mu^{-1}e_{d},n]= \mu^{-1}[e_{d},n]- \nabla_{n}\mu^{-1}e_{d}=\mu^{-1}\Sigma a_{dj}e_{j} - \nabla_{n}\mu^{-1}e_{d}.$$
\begin{align*}
  2\Pi(\eta^{-1}e_{1}, \mu^{-1}e_{d})&= \mu^{-1}\eta^{-1}g([e_{1},n],e_{d})- \mu^{-1}\nabla_{n}\eta^{-1}g(e_{1},e_{d})
+\mu^{-1}\eta^{-1}g([e_{d},n],e_{1})\\
&\hspace{.4cm}- \eta^{-1}\nabla_{n}\mu^{-1}g(e_{d},e_{1})\\
&=\mu^{-1}\eta^{-1}g([e_{1},n],e_{d})+ g([e_{d},n],e_{1}))\\
&=\mu^{-1}\eta^{-1} (\Sigma a_{ij}e_{j},e_{d})+ g(\Sigma a_{dj}e_{j},e_{1}))\\
&=\mu^{-1}\eta^{-1}(a_{1d} g(e_{d},e_{d})+ a_{d1}g(e_{1},e_{1}))\\
&= \mu^{-1}\eta^{-1}a_{d1}(-g(e_{d},e_{d})+ g(e_{1},e_{1})).
\end{align*}

Since $[e_{d},n] \in span<e_{d},n>$, then $a_{d1}=...=a_{d-1}=0$. We thus conclude that
 $$\Pi (\eta^{-1}e_{1},\mu^{-1}e_{d})=0.$$
\end{proof}\qed

We consider $ \mathbb{C}^{m}\setminus{0}\,$ with K\"ahler 
 metric $g= \partial \overline {\partial} f(S)=(f_{S}\delta_{ij}+f_{SS}\bar{z_{i}}z_{j})dz_{i} \wedge d\bar{z_{j}}$, $\Sigma=\{(z_1,z_2,\dots,z_m) \in \C^m: S=R^2=|z_1|^2+ |z_2|^2+\dots+|z_m|^2 \} \subset \C^m$, the normal vector $n$ and $J(n)=in$ and moreover an orthonormal basis $e_{1},...,e_{2m-2}$ for $<n,J(n)>^{\perp}$. Let $e_{2m-1}=J(n)$ and $e_{2m}=n$ , the metric g is written by:
\[
\begin{bmatrix} f_{S}I_{2m-2}& 0\\
0 & (f_{S}+f_{SS}S)I_{2}
\end{bmatrix}.
\]
\medskip

    \begin{Thm}
    
        The principal curvatures of the family $\Sigma_{S}^{2m-1} \subset\mathbb{C}^{m}\setminus{0}$  with a $U(m)$-invariant K\"ahler metric $\omega = \sqrt{-1} \partial \overline {\partial} f(S)$  are as follows:
    
  \[
    \lambda_{1}=  \lambda_{2}=...= \lambda_{2m-2} =- \frac{ \sqrt{f_{S}+f_{SS}S}}{f_{S}\sqrt{S}}, \hspace{.1cm} \lambda_{2m-1}=- \frac{f_{S}+ 3Sf_{SS}+ S^{2}f_{SSS}}{(f_{S}+f_{SS}S)^{\frac{3}{2}} \sqrt{S}}.
 \] 
    
 where $S=\Sigma_{i=1}^{m}|z_{i}|^2 $.
 
\end{Thm}

\begin{proof}
In the setting of the Proposition \ref{2nd fundamental form 1}, we have $\Sigma= S^{2m-1}(r)$ and $M=\mathbb{C}^{m}\setminus{0}$. Furthermore, we have $A^2=\mu^2=f_{S}+f_{SS}S$ and $\eta^{2}=f_{S}$. Additionally we get
$\eta^{-1}\nabla_{n}\eta= \frac{S}{f_{S}}f_{SS}$ and $\mu^{-1}\nabla_{n}\mu= \frac{\sqrt{S}}{\mu^{2}}(2f_{SS}+ Sf_{SSS})$.

 Now by computing $g^{-1}\Pi(g)$, we obtain the following principal curvatures:
\[
\lambda_{1}=...=\lambda_{2m-2}=- \frac{\sqrt{f_{S}+Sf_{SS}}}{f_{S}\sqrt{S}}, \hspace{.1cm}\lambda_{2m-1}=- \frac{(f_{S}+3Sf_{SS}+S^{2}f_{SSS})}{(f_{S}+f_{SS}S)^{\frac{3}{2}} \sqrt{S}}.
\]
    
\end{proof}\qed

\section{Mean Curvature Flow}
\label{sec:5}
\hspace{10pt} In this section, we prove our main theorem presenting the mean curvature flow with initial data given by a special class of hyperspheres in $\mathbb{C}^{m}$ with a $U(m)$-invariant K\"ahler metrics. To do so, we first compute the mean curvature which is the sum of the eigenvalues of the second fundamental form.
\medskip

\begin{Thm}\label{Thm :mean :curvature}
The mean curvature of the family $\Sigma_{S}^{2m-1} \subset \mathbb{C}^{m}\setminus{0}$  with $U(m)$-invariant K\"ahler metric $\omega=\sqrt{-1}\partial \bd f(S)$ is given as follows:
\begin{equation*}
     H(S)=\frac{-1}{(2m-1)(f_{S}+Sf_{SS})^{\frac{3}{2}}\sqrt{S}f_{S}}((2m-2)(f_{S}+Sf_{SS})^2+f_{S}(f_{SSS}S^2+3Sf_{SS}+f_{S})).
\end{equation*}
\end{Thm}

In the following two lemmas, we compute the square of the second fundamental form to see whether the mean curvature flow contains singularity or not.
\medskip

\begin{lem}
  Let $A$ be the second fundamental form of the family of $\Sigma_{S}^{2m-1} \subset  \mathbb{C}^{m} \setminus {0} $  with $U(m)$-invariant K\"ahler metric $\omega = \sqrt{-1}  \partial \overline {\partial} f(S)$.  Then the square of its norm, $|A|^{2}$ is as follows:
\[
\frac{(2n-2)(f_{S}+f_{SS}S)^{4}+ f_{S}^{2} (f_{S}+ 3Sf_{SS}+ S^{2}f_{SSS})^{2}}{f_{S}^{2}  (f_{S}+f_{SS}S)^{3} S}
\]  
\end{lem}

\begin{proof}
     The principal curvatures for the hyperspheres are:
 
\[
    \lambda_{1}=  \lambda_{2}=...= \lambda_{2m-2} =- \frac{ \sqrt{f_{S}+f_{SS}S}}{f_{S}\sqrt{S}}, \hspace{.1cm} \lambda_{2m-1}=- \frac{f_{S}+ 3Sf_{SS}+ S^{2}f_{SSS}}{(f_{S}+f_{SS}S)^{\frac{3}{2}} \sqrt{S}}.
\] 

Now we can compute $|A|^{2}$ as following:

\medskip
\hspace{1.9cm}$|A|^{2}  =  \lambda_{1}^{2}+ \lambda_{2}^{2}+ ...+ \lambda_{2m-1}^{2} = (2m-2)\lambda_{1}^{2}+ \lambda_{2m-1}^{2}$
 
$$ =  \frac{(2m-2)(f_{S}+f_{SS}S)^{4}+ f_{S}^{2} (f_{S}+ 3Sf_{SS}+ S^{2}f_{SSS})^{2}}{f_{S}^{2}  (f_{S}+f_{SS}S)^{3} S}\,.$$
\end{proof}\qed

\begin{lem}For each $g$ with the following conditions,
$$
g_{S}(0) > 0,\hspace{.1cm}\frac{1}{S} + g_{S} > 0,\hspace{.1cm} \text{and } g_{S} +Sg_{SS} > 0,
$$

 $ |A|^{2} $ blows up only at $S=0$.
 
 \end{lem}

\begin{proof} 
\[
|A|^{2} = \frac{(2n-2)(g_{S}+g_{SS}S)^{4}+ (\frac{1}{S}+g_{S})^{2} (g_{S}+ 3Sg_{SS}+ S^{2}g_{SSS})^{2}}{(\frac{1}{S}+g_{S})^{2}  (g_{S}+g_{SS}S)^{3} S}
\]

We know that $g$ is a smooth function and does not blow up. When $S=0$, the numerator is always positive by the above conditions of $g$. Thus the singularity only happens when $S=0.$ \qed
\end{proof}

Now we prove the main result of our work in the next theorem in which we investigate the mean curvature flow for our setting.
\medskip

\begin{Thm}
Consider $\mathbb{C}^{m} \setminus {0}$ with a $U(m)$-invariant K\"ahler metric $\omega = \sqrt{-1} \partial \overline {\partial} f(S)$ where $f(S)= \log S + g(S)$ and $g$ is an analytic function with the following conditions: 
\[
g_{S}(0) > 0,\hspace{.1cm}\frac{1}{S} + g_{S} > 0,\hspace{.1cm} \text{and } g_{S} +Sg_{SS} > 0.
\]
There exists $\epsilon > 0$ such that if  $R_{o} < \epsilon$, we can choose one hypersphere with radius $R_{0}$ in such a way that the mean curvature flow with initial condition $\Sigma_{R(0)} = \Sigma_{R_{0}}$ converges to the exceptional divisor at a finite time and we have a singularity of Type I.
\end{Thm}
 \begin{proof}
 
 The mean curvature flow problem for the hyperspheres $\Sigma_{S}$  is the following ordinary differential equation (ODE):
 \[
 \frac{ dR(t)}{dt}= H(R(t)).
 \]
 
 We can choose $\epsilon > 0 $ such that if we start the flow with the initial data $R(0)=R_{0}<\epsilon$, the mean curvature does not vanish and is negative. In the previous lemma we observe that there is only one singularity at $R(t)=0$. Therefore, the time of singularity $(T_{sing})$ happens whenever $R(t)=0$.  This means that if the flow starts at $t=0$, then $ |A|^{2} $ is bounded for all $t \in [0,T_{sing})$.
 We can write the mean curvature flow problem as $\frac{ dR(t)}{dt}= \frac{1}{R^\alpha(t)}K(R(t))$ for some $\alpha >0$, where $K(R(t))$ is an analytic function without singularity and its Taylor series near $R(t)=0$ is as follows: $K(R(t)) = \sum_{n=0}^{\infty} \frac{K^{n}(0)}{n!} R^{n}(t)$. We thus get 
 $R^{\alpha} (t) \frac{ dR(t)}{dt}=\sum_{n=0}^{\infty} \frac{K^{n}(0)}{n!} R^{n}(t)$. By applying integral on both sides, we get $\frac{1}{\alpha+1} R^{\alpha+1}(t)= K(0)t + \sum_{n=1}^{\infty}  \frac{K^{n}(0)}{(n+1)!} R^{n+1}(t) + C$ for some constant C. Moreover, with initial condition $R(0)=R_{0}$ we have $C=\frac{R_{0}^{\alpha+1}}{\alpha+1} -  \sum_{n=1}^{\infty}  \frac{K^{n}(0)}{(n+1)!} R_{0}^{n+1}$. Further, we have the singularity only at $R(t)=0$. Hence
 $$T_{sing} = \frac{1}{K(0)} ( \sum_{n=1}^{\infty}  \frac{K^{n}(0)}{(n+1)!} R_{0}^{n+1} - \frac{R_{0}^{\alpha+1}}{\alpha+1}  ).$$
 We can easily conclude that the time of singularity is finite and we can employ the Propositions \ref{Pro:1} and \ref{Pro:2}. Since these Propositions are well-known local theorems, we can apply them in Riemannian case. Therefore, we can conclude that the flow does not stop (i.e., keep restarting) and converges to $R(t)=0$ which is the exceptional divisor in $Bl_{0}\mathbb{C}^{m}$. Moreover,  We can write the square of the second fundamental form  as $|A|^{2}=\frac{W(R(t))}{R^{2}(t)}$,  where $W(R(t))$ is an analytic function without singularity. Its Taylor series then near $R(t)=0$ is $W(R(t)) = \sum_{n=0}^{\infty} \frac{W^{n}(0)}{n!} R^{n}(t)$.
 Clearly we have
 \[
 |A|^{2}= \frac{W^{0}(0)}{R^{2}(t)}+ \frac{W^{1}(0)}{R(t)}+\frac{W^{2}(0)}{2}+ \sum_{n=3}^{\infty} \frac{W^{n}(0)}{n!} R^{n-2}(t).
 \]
 Therefore we get
  \[
\lim_{t \to\ T_{sing}}(T_{sing}-t)|A|^{2}= \lim_{t \to\ T_{sing}}(T_{sing}-t)\frac{W^{0}(0)}{R^{2}(t)}+\lim_{t \to\ T_{sing}}(T_{sing}-t)\frac{W^{1}(0)}{R(t)}
\]
$$
+\lim_{t \to\ T_{sing}}(T_{sing}-t)\frac{W^{2}(0)}{2}
+\lim_{t \to\ T_{sing}}(T_{sing}-t)\sum_{n=3}^{\infty} \frac{W^{n}(0)}{n!} R^{n-2}(t).
$$
$R(t)$ goes to zero as $t$ goes to $T_{sing}$, so we can easily check that 
 $$ \lim_{t \to\ T_{sing}}(T_{sing}-t)\frac{W^{2}(0)}{2}=\lim_{t \to\ T_{sing}}(T_{sing}-t)\sum_{n=3}^{\infty} \frac{W^{n}(0)}{n!} R^{n-2}(t)=0.$$
 Since $\frac{ dR(t)}{dt}= H(R(t))$, we have 
 $R^{\prime}(T_{sing})= \frac{ dR(t)}{dt}|_{t=T_{sing}}=H(R(T_{sing}))=H(0)=\infty.$
 By using Hopital method we can conclude that 
 $$\lim_{t \to\ T_{sing}}(T_{sing}-t)\frac{W^{1}(0)}{R(t)}=\lim_{t \to\ T_{sing}}\frac{-W^{1}(0)}{R^{\prime}(t)}<\infty .$$
We can also compute that $\lim_{t \to\ T_{sing}}R(t)H(R(t)) \neq 0$. Again by using Hopital method we can easily see that 
$$ \lim_{t \to\ T_{sing}}(T_{sing}-t)\frac{W^{0}(0)}{R^{2}(t)}< \infty.$$
 Consequently, 
 $$\lim_{t \to\ T_{sing}}(T_{sing}-t)max|A|^2 < \infty.$$ 
 The singularity is thus Type I. 

\end{proof}\qed

The assumption of analyticity in the above Theorem is not restrictive. Many interesting
Kähler metrics are analytic. For example, as proved by Hopf and Morrey constant scalar curvature
Kähler metrics satisfy this hyphothesis \cite{Morrey}.
\medskip
\begin{re}
 We can observe that when $S(t) \to 0$, then $\lambda_{1}=...=\lambda_{2m-2} \to 0 $ and $\lambda_{2m-1}\to \infty$. This means that when $S(t) \to 0$,  one of the principal directions collapses and the hypersphere converges to the exceptional divisor, which is holomorphic submanifold of  $Bl_{0}\mathbb{C}^{m}$. Since holomorphic submanifolds of complex manifolds are minimal, so one would naturally expect that the principal curvature vanishes there.
 
\end{re}

In some examples we can estimate $\epsilon$ as $+\infty$ including the Burns metric. Example \ref{eg:1} provides an instance of the mean curvature flow problem for the Burns metric.
\medskip

\begin{ex}
\label{eg:1}
    Consider $Bl_{0}\mathbb{C}^{2}$ with the Burns metric given by $\omega =\sqrt{-1}  
    \partial \overline {\partial} (\log(S)+ S)$.
    We can choose an arbitrary hypersphere $\Sigma_{R_{0}}$ as initial condition for the mean curvature flow. The mean curvature flow of the hypersphere converges to $S^{2}$ at a finite time and we have the singularity of Type I.
\end{ex}    
     
\begin{proof}

The mean curvature flow problem for the hyperspheres $\Sigma_{S}$  is the following ODE:

 \[
 \frac{ dR(t)}{dt}= H(R(t)).
 \]

\medskip
Now the principal curvatures of $\Sigma_{S}$ with Burns metric are:
\[
\lambda_{1}=\lambda_{2}= \frac{-R}{(R^{2}+1)}, \hspace{.2cm} \lambda_{3}=\frac{-1}{R}.
\]

\medskip
Moreover, the mean curvature of these families and $|A|^{2}$ are given by
\[
H(R(t))= \frac{-1}{3}\frac{3R^{2}(t)+1}{R(t)(R^{2}(t)+1)}, \hspace{.1cm}|A|^{2}= \frac{2R^{4}(t)+(R^{2}(t)+1)^{2}}{R^{2}(t)(R^{2}(t)+1)^{2}}.
\]

\medskip
Therefore the mean curvature problem is equivalent to:
\[
\frac{ dR(t)}{dt} = \frac{-1}{3}\frac{3R^{2}(t)+1}{R(t)(R^{2}(t)+1)}.
\]
\medskip
The solution of the equation with initial data $R(0)=R_{0}$ would be

\[
\frac{R^{2}(t)}{2} + \frac{1}{3} \log(3R^{2}(t)+1)=-t+c.
\]
\medskip
$|A|^{2}$ blows up only when $R(t)=0$. With the initial condition $R(0)=R_{0}$ we get the time of singularity as following:
\[
T_{sing}= \frac{R_{0}^{2}}{2} + \frac{1}{3} \log(3R_{0}^{2}+1).
\]

\medskip
The time of singularity is finite and the flow exists for all $t \in [0 , T_{sing})$. We can also check that there exists a positive constant $C$ such that $|A|^{2} < \frac{C}{|T_{sing}-t|}$. The singularity is thus Type I.

\end{proof}\qed

\section*{Acknowledgement}
We thank Professor Claudio Arezzo for many valuable discussions and comments about this work. His insightful feedback brought our work to a higher level. We are also greatly indebted with Professor Reza Seyyedali. He also contributed to improve our paper by kindly providing several comments on this paper.

%%%%%%%%%

\begin{itemize}
    \item Tarbiat Modares University Tehran and ICTP Trieste, farnazghanbari@modares.ac.ir
    \item SISSA Trieste, samreena01@gmail.com
\end{itemize}


\begin{thebibliography}{999}

 \bibitem{Brakke}
 K. Brakke – “The motion of a surface by its mean curvature”, Princeton University Press. (1978).

 \bibitem{Chen}
Y.-G. Chen, Y. Giga, S. Goto – “Uniqueness and existence of viscosity solutions
of generalized mean curvature flow equations”, J. Differential Geom. 33 (1991), p. 749–786, announcement in Proc. Japan. Acad. Ser. A 65 (1985) 207–210.


\bibitem{Evans3}
  L. Evans, J. Spruck – “Motion of level sets by mean curvature iv.”, J. Geom. Anal. 5 (1995), p. 77–
114.

\bibitem{Evans2}
 L. Evans, J. Spruck – “Motion of level sets by mean curvature iii.”, J. Geom. Anal. 2 (1992), p. 121–
150.

\bibitem{Evans1}
 L. Evans, J. Spruck – “Motion of level sets by mean curvature ii.”, Trans. Amer. Math. Soc. 330
(1992), p. 321–332.


 \bibitem{Evans}
 L. Evans, J. Spruck – “Motion of level sets by mean curvature i.”, J. Differential Geom. 33 (1991), p. 635–681. 


\bibitem{Huisken:1986}
G. Huisken – “Contracting convex hypersurfaces in Riemannian manifolds by their mean curvature”, Invent. Math. 84 (1986), 462-480.

\bibitem{Huisken: 1984}
G. Huisken – “Flow by mean curvature of convex surfaces into spheres”, J. Differential
Geom. 20 (1984), p. 237–266.

\bibitem{Ilmanen}
T. Ilmanen – “Convergence of the allen-cahn equation to brakke’s motion by mean
curvature”, J. Differential Geom. 38 (1993), p. 417–461.






\bibitem{Mantegazza}
C. Mantegazza – “Lecture Notes on Mean Curvature Flow”, Springer Science \& Business Media, 290 (2011).

\bibitem{Morrey}
C.B. Morrey – “On the analyticity of the solutions of analytic non-linear elliptic systems of partial differential equations. I. Analyticity in the interior”, American Journal of Mathematics. (1958).

\bibitem{Zhu}
X. Zhu – “Lectures on mean curvature flows”, American Mathematical Soc. 32 (2002).

\end{thebibliography}
\end{document}